\documentclass[letterpaper,12pt]{article}

\usepackage[english]{babel}
\usepackage[utf8x]{inputenc}
\usepackage[T1]{fontenc}
\usepackage[affil-it]{authblk}

\usepackage[letterpaper,top=3cm,bottom=3cm,left=3cm,right=3cm,marginparwidth=1.75cm]{geometry}

\usepackage{amsmath}
\usepackage{amsthm}
\usepackage{amssymb}
\usepackage[all,cmtip]{xy}
\usepackage{graphicx}
\usepackage[colorinlistoftodos]{todonotes}
\usepackage[colorlinks=true, allcolors=blue]{hyperref}
\usepackage{tikz-cd}
\usepackage{theoremref}

\theoremstyle{plain}
\newtheorem{theorem}{Theorem}[section]

\newtheorem{prop}[theorem]{Proposition}

\theoremstyle{definition}

\newtheorem{rem}[theorem]{Remark}
\usepackage{theoremref}

\newcommand{\Z}{\mathbb{Z}}

\newcommand{\C}{\mathbb{C}}
\newcommand{\BAG}{B^A(G)}
\newcommand{\BAH}{B^A(H)}
\newcommand{\gBAG}{\widetilde{B^A}(G)}

\newcommand{\qBAG}{\overline{B^A}(G)}
\newcommand{\qBAH}{\overline{B^A}(H)}
\newcommand{\MAG}{\mathcal{M}^A_G}
\newcommand{\MAH}{\mathcal{M}^A_H}
\newcommand{\SG}{\mathcal{S}_G}
\newcommand{\SH}{\mathcal{S}_H}
\newcommand{\thetaS}{\theta_{\mathcal{S}}}
\newcommand{\thetaBS}{\theta_{[\mathcal{S}]}}

\newcommand{\lexp}[2]{\setbox0=\hbox{$#2$} \setbox1=\vbox to
                 \ht0{}\,\box1^{#1}\!#2}
\newcommand{\Hom}{\mathrm{Hom}}

\title{Groups with isomorphic fibered Burnside rings}

\author{
Robert Boltje\\
Department of Mathematics\\
University of California\\
Santa Cruz, CA 95064\\
U.S.A\\
\texttt{boltje@ucsc.edu}
\and 
Benjam\'in Garc\'ia\\
Centro de Ciencias Matem\'aticas\\
Universidad Nacional Aut\'onoma de M\'exico\\
Morelia, Michoac\'an 58089\\
M\'exico\\
\texttt{bagh1704@hotmail.com}}

\begin{document}

\maketitle

\begin{abstract}
Let $G$ and $H$ be finite groups. We give a condition on $G$ and $H$ that implies that the $A$-fibered Burnside rings $B^A(G)$ and $B^A(H)$ are isomorphic. As a consequence, we show the existence of non-isomorphic groups $G$ and $H$ such that $B^A(G)$ and $B^A(H)$ are isomorphic rings. Here, the abelian fiber group $A$ can be chosen in a non-trivial way, that is, such that $B^A(G)$ and $B^A(H)$ are strictly bigger than the Burnside rings of $G$ and $H$, for which such counterexamples are already known.


\end{abstract}

\section{Introduction}

If $A$ is an abelian group, the \textit{$A$-fibered Burnside ring} of a finite group $G$, denoted by $\BAG$, is the Grothendieck ring of the category of $A$-fibered $G$-sets. An $A$-fibered $G$-set is a $G\times A$-set which is free as $A$-set and has finitely many $A$-orbits. The category of $A$-fibered $G$-sets has coproducts and is symmetric monoidal with respect to a tensor product $\otimes_A$. These operations yield the ring structure on $\BAG$.
The $A$-fibered Burnside ring was first introduced in greater generality by Dress in \cite{Dress}, where an action of the group $G$ on $A$ was allowed and the notation $\Omega(G,A)$ was used. We restrict ourselves to the case of the trivial action of $G$ on $A$ and use the notation $B^A(G)$.

The case $A=k^{\times}$ for a field $k$ is particularly interesting, since in this case the ring $B^{k^\times}(G)$ is also isomorphic to the Grothendieck ring $D^k(G)$ of the category of finite $G$-line bundles over $k$ (introduced in \cite{Boltje2001} and denoted by $R_{k+}^{\mathrm{ab}}(G)$), which is also sometimes called the category of monomial $k$-representations of $G$ (see \cite[Section~5]{Boltje2001}). There is a direct connection to $kG$-modules: One has a natural ring homomorphism $D^k(G)\to R_k(G)$ into the Grothendieck ring of finitely generated $kG$-modules which is surjective if $k$ is algebraically closed. For $k=\C$, a natural section of $D^{\C}(G)\to R_{\C}(G)$, i.e., a canonical induction formula, was given in \cite{Bolt2}, and a categorification of the canonical induction formula (a monomial resolution in terms of $G$-line bundles over $\C$) was given in \cite{Boltje2001}.

Recalling that $\BAG \cong B(G)$, the Burnside ring of $G$, if and only if $A$ has trivial $|G|$-torsion, we will say that two non-isomorphic finite groups $G$ and $H$ provide a \textit{non-trivial counterexample to the isomorphism problem of the $A$-fibered Burnside ring} if $\BAG$ and $\BAH$ are isomorphic as rings and there is a non-trivial element $a$ in $A$ such that $a^{|G|}=1$. Examples of non-isomorphic finite groups $G$ and $H$ with isomorphic Burnside rings have already been given for instance by Th\'evenaz in \cite{Thev}.

The notion of a \textit{species isomorphism} for fibered Burnside rings was introduced by the second author in \cite{Gar1} for a ring isomorphism preserving the standard bases given by conjugacy classes of monomial pairs, which is analogous to an isomorphism of mark tables in the case of Burnside rings. In Section 2 of this note we present a sufficient condition on finite groups $G$ and $H$ for the existence of a species isomorphism between their fibered Burnside rings. In Section 3, we use this result to prove that Thévenaz' counterexamples to the isomorphism problem for Burnside rings (see \cite{Thev}) of order $p^2q$ for primes $p$ and $q$ such that $q|(p-1)$, provide also non-trivial counterexamples for the $A$-fibered Burnside ring when the fiber $A$ has trivial $p$-torsion and elements of order $q$.

\subsection*{Notation}
Throughout this note, the letters $G$ and $H$ stand for finite groups. 
We write $\SG$ for the set of subgroups of $G$, and $[\SG]\subseteq \SG$ for a set of representatives of the conjugacy classes. For an element $g\in G$, we denote the resulting conjugation map $c_g$ also by ${^g\_}\colon G\to G,\ x\mapsto gxg^{-1}$. For subgroups $K$ and $L$ of $G$ we write $K=_G L$ and $K\le_G L$ if there exists $g\in G$ with $K=\lexp{g}{L}$ and $K\le \lexp{g}{L}$, respectively. If $G$ acts on a set $X$, we write $G\backslash X$ for the set of its orbits and we write $x=_Gy$ if two elements $x$ and $y$ of $X$ are in the same orbit.

\section{A criterion for species isomorphisms}

We first recall some basic definitions and results on fibered Burnside rings. We refer the reader to \cite{BY}, \cite{Dress} and \cite{Gar1} for further details. 

Let $\MAG$ denote the set of all pairs $(K,\phi)$, also called \textit{subcharacters} or \textit{monomial pairs}, where $K\leq G$ and $\phi\colon K\longrightarrow A$ is a group homomorphism. 
The group $G$ acts on $\MAG$ by $\lexp{g}{(K,\phi)}:=(\lexp{g}{K},\lexp{g}{\phi})$, where $\lexp{g}{\phi}(x):=\phi(g^{-1}xg)$ for $g\in G$, $(K,\phi)\in\MAG$ and $x\in \lexp{g}{K}$.
By \cite[Proposition~2.2]{Dress}, the transitive $A$-fibered $G$-sets are parametrized by the $G$-orbits of $\MAG$. 
More precisely, if $X$ is a transitive $A$-fibered $G$-set, one associates to it the pair $(K,\phi)$ consisting of the $G$-stabilizer $K$ of the $A$-orbit $Ax$ of a fixed element $x\in X$ and the homomorphism $\phi\colon K\to A$ defined by $gx=\phi(g)x$ for all $g\in K$. The orbit of $(K,\phi)$ does not depend on the choice of $x$ and is denoted by $[K,\phi]_G$. Thus, $\BAG$ can be regarded as the free abelian group with basis $G\backslash \MAG$. By \cite[2.3]{Dress}, the multiplication is given by 
$$[K,\phi]_G\cdot [L,\psi]_G=
\sum_{KsL\in K\backslash G/ L}\bigl[K\cap {^sL},\phi|_{K\cap {^sL}}{^s\psi}|_{K\cap {^sL}}\bigr]_G\,,$$
with identity element $[G,1]_G$. That this multiplication yields a commutative ring structure on $\BAG$ follows from the properties of the tensor product $\otimes_A$ of $A$-fibered $G$-sets, see \cite[Section~1]{Dress}.

\smallskip
For $K\le G$, $\Z\Hom(K,A)$ is the group ring of $\Hom(K,A)$ which is a group under point-wise multiplication. The group $G$ acts by conjugation on the resulting product ring $\prod_{K\leq G}\Z \Hom(K,A)$. More precisely, for $g\in G$ and $(x_K)_{K\le G}\in\prod_{K\le G}\Z\Hom(K,A)$, one has
$$\lexp{g}{(x_K)_{K\le G}}:=\bigl(\lexp{g}{x_{\lexp{g^{-1}}{K}}}\bigr)_{K\le G}\,,$$
where, for $L\le G$ and $y=\sum_\phi a_\phi \phi\in\Z\Hom(L,A)$, we set $\lexp{g}{y}:=\sum_\phi a_\phi\lexp{g}{\phi}\in\Z\Hom(\lexp{g}{L},A)$.
Thus, if for $(K,\phi)\in\MAG$ the element $b_{(K,\phi)}\in \prod_{K\le G}\Z\Hom(K,A)$ is defined as $b_{(K,\phi)}=(x_L)_{L\le G}$ with $x_L=\phi$ if $L=K$ and $x_L=0$ otherwise, then these elements form a $\Z$-basis of $\prod_{K\le G}\Z\Hom(K,A)$ and $\lexp{g}{b_{(K,\phi)}}=b_{\lexp{g}{(K,\phi)}}$. 

\smallskip
The \textit{ghost ring} of $\BAG$ is defined as the subring $\widetilde{B^A}(G):=\left(\prod_{K\leq G}\Z \Hom(K,A)\right)^G$ of $G$-fixed points under this action. Since the elements $b_{(K,\phi)}$, $(K,\phi)\in\MAG$, form a $\Z$-basis of $\prod_{K\le G}\Z\Hom(K,A)$ that is permuted by $G$, their orbit sums $\widetilde{b}_{(K,\phi)}:=\sum_{g\in [G/N_G(K,\phi)]} \lexp{g}{b_{(K,\phi)}}$, where $(K,\phi)$ runs through a set of representatives of the $G$-orbits of $\MAG$ and $N_G(K,\phi)$ denotes the $G$-stabilizer of $(K,\phi)$, form a $\Z$-basis of $\widetilde{B^A}(G)$.

\smallskip
The natural projection map $\pi_{[\SG]}:\prod_{K\leq G}\Z \Hom(K,A)\longrightarrow \prod_{K\in [\SG]}\Z \Hom(K,A)$ is a ring homomorphism, and it is injective when restricted to $\gBAG$,  since the components indexed by elements outside $[\SG]$ are determined by the components indexed by elements in $[\SG]$. Moreover, one has
$$\gBAG\cong \pi_{[\SG]}\left(\gBAG\right)=\prod_{K\in [\SG]}\left(\Z \Hom(K,A)\right)^{N_G(K)}\,.$$
In fact, each $K$-component of a $G$-fixed point of $\prod_{K\le G}\Z\Hom(K,A)$ must be an $N_G(K)$-fixed point. For the converse, note that for $K\in[\SG]$, the $N_G(K)$-orbit sums of the elements $\phi\in\Hom(K,A)$ form a $\Z$-basis of $(\Z\Hom(K,A))^{N_G(K)}$ and that the $N_G(K)$-orbit sum of $\phi$ is equal to the $K$-component of $\widetilde{b}_{(K,\phi)}$.
We set $\qBAG:=\prod_{K\in [\SG]}\left(\Z \Hom(K,A)\right)^{N_G(K)}$ and, for $K\in[\SG]$ and $(K,\phi)\in\MAG$, we set $\overline{b}_{(K,\phi)}:=\pi_{[\SG]}\left(\widetilde{b}_{(K,\phi)}\right)$. Thus, the $K$-component of $\overline{b}_{(K,\phi)}$ is equal to the $N_G(K)$-orbit sum of $\phi$ and its $L$-component is $0$ for all $L\in[\SG]$ with $L\neq K$.

\smallskip
By \cite[2.5]{Dress}, the map
$$\Phi^A_G\colon B^A(G)\longrightarrow \widetilde{B^A}(G)\,,\quad [L,\psi]_G\mapsto 
\left(\sum_{\phi\in\Hom(K,A)}\gamma^G_{(K,\phi),(L.\psi)}\phi\right)_{K\leq G}\,,$$
is an injective ring homomorphism, also known as the \textit{mark morphism}, 
where
$$\gamma^G_{(K,\phi),(L,\psi)}=|\{sL\in G/L\;|\; (K,\phi)\leq {^s(L,\psi)} \}|\,,$$
for $(K,\phi),(L,\psi)\in \MAG$. Several properties of these numbers are listed in \cite[Section 1]{Bolt} and in \cite[Lemma 2.2]{Gar1}. We set $\overline{\Phi}^A_G:=\pi_{[\SG]}\circ\Phi^A_G\colon \BAG\to\prod_{K\in[\SG]} \bigl(\Z\Hom(K,A)\bigr)^{N_G(K)}$.

\smallskip
If $H$ is another finite group, a ring isomorphism $\Theta:\BAG\longrightarrow \BAH$ is called a \textit{species isomorphism} if $\Theta([K,\phi]_G)=[R,\rho]_H\in H\backslash \MAH$ for every $[K,\phi]_G\in G\backslash \MAG$, see \cite[Def. 3.1]{Gar1}. We will make use of the following theorem which is part of the statement of Theorem~3.14 in \cite{Gar1}.

\begin{theorem}\thlabel{Previous Species Thm}
Let $A$ be an abelian group and let $G$ and $H$ be finite groups. There exists a species isomorphism from $B^A(G)$ to $B^A(H)$ if and only if there exist bijections $\thetaS\colon \SG\longrightarrow \SH$ and $\theta_K\colon \Hom(K,A)\longrightarrow\Hom(\thetaS(K),A)$, for $K\le G$, satisfying the following two conditions:

\smallskip
{\rm (a)} $\gamma_{(\thetaS(K),\theta_K(\phi)), (\thetaS(L),\theta_L(\psi))}^H = \gamma_{(K,\phi),(L,\psi)}^G $ for all $(K,\phi),(L,\psi)\in\MAG$.

\smallskip
{\rm (b)} The group homomorphism $\widetilde{\Theta}\colon \widetilde{B^A}(G) \longrightarrow \widetilde{B^A}(H)$ determined by mapping $\widetilde{b}_{(K,\phi)}$ to $\widetilde{b}_{(\thetaS(K),\theta_K(\phi))}$, for $(K,\phi)\in\MAG$, is a ring isomorphism.
\end{theorem}

\begin{rem}\label{old remark}
In order to clarify the statement in Theorem~\ref{Previous Species Thm}, we will show that the map $\widetilde{\Theta}$ in (b) is well-defined, provided that the condition in (a) holds.  That is, if $(K,\phi), (L,\psi) \in \MAG$ are $G$-conjugate, then 
$\widetilde{b}_{(\thetaS(K),\theta_K(\phi))}=\widetilde{b}_{(\thetaS(L),\theta_L(\psi))}\in \widetilde{B^A}(H)$.  For this it suffices to show that $(\thetaS(K),\theta_K(\phi))$ and $(\thetaS(L), \theta_{L}(\psi))$ are $H$-conjugate. By part 2 of \cite[Lemma 2.2]{Gar1}, $(K,\phi)=_G(L,\psi)$ if and only if $0\neq \gamma^G_{(K,\phi),(L,\psi)}$ and $0\neq \gamma^G_{(L,\psi),(K,\phi)}$. Thus, the condition in (a) implies that $(K,\phi)=_G(L,\psi)$ if and only if $(\thetaS(K),\theta_K(\phi))=_H(\thetaS(L), \theta_{L}(\psi))$. 

For later use in the proof of Theorem~\ref{SpeciesThm} we note that if $K=_GL$ then $(K,1)=_G(L,1)$ and, by the above, $(\thetaS(K),\theta_K(1))=_H(\thetaS(L),\theta_L(1))$, which implies that $\thetaS(K)=_H\thetaS(L)$. Conversely, if $\thetaS(K)=_H\thetaS(L)$, then $(K,\theta_K^{-1}(1))=_G(L,\theta_L^{-1}(1))$, implying $K=_GL$.  Thus, the condition in (a) also implies that $K=_GL$ if and only if $\thetaS(K)=_H\thetaS(L)$.
\end{rem}

Next we prove a slight modification of the above theorem.

\begin{theorem} \thlabel{SpeciesThm}
Let $A$ be an abelian group and let $G$ and $H$ be finite groups. Then there exists a species isomorphism from $\BAG$ to $\BAH$ if and only if for any given $[\SG]$ and $[\SH]$, there are bijections $\thetaBS\colon [\SG]\longrightarrow [\SH]$ and $\theta_K\colon\Hom(K,A)\longrightarrow \Hom(\thetaBS(K),A)$ for $K\in [\SG]$, satisfying the following two conditions:

\smallskip
{\rm (a)} $\gamma^H_{\left(\thetaBS(K),\theta_{K}(\phi)\right),\left(\thetaBS(L),\theta_L(\psi)\right)}=\gamma^G_{\left(K,\phi\right),\left(L,\psi\right)}$ for all $K,L\in [\SG]$, $\phi\in \Hom(K,A)$ and $\psi\in \Hom(L,A)$.

\smallskip
{\rm (b)} The map $\overline{\Theta}\colon \qBAG \longrightarrow \qBAH$, $\overline{b}_{(K,\phi)}\mapsto \overline{b}_{(\thetaBS(K),\theta_K(\phi))}$, for $K\in[\SG]$ and $\phi\in \Hom(K,A)$, is a ring isomorphism.
\end{theorem}

\begin{rem}\label{new remark}
With the same arguments as in Remark~\ref{old remark} one can show that in the situation of Theorem~\ref{SpeciesThm} the condition in (a) implies that the map $\overline{\Theta}$ in (b) is well-defined. More precisely, (a) implies that for any $K\in[\SG]$, $\phi,\psi\in\Hom(K,A)$, one has $(K,\phi)=_G(K,\psi)$ if and only if $(\thetaS(K),\theta_K(\phi))=_H(\thetaS(K),\theta_K(\psi))$. 
\end{rem}

\medskip\noindent
{\it Proof of Theorem~\ref{SpeciesThm}.}\quad
First suppose that there exists a species isomorphism from $\BAG$ to $\BAH$. Then there exist bijections $\thetaS$ and $\theta_K$ for $K\le G$ as in Theorem~\ref{Previous Species Thm}, satisfying the conditions (a) and (b) in Theorem~\ref{Previous Species Thm}. Let $[\SG]$ and $[\SH]$ be given. Then for every $K\in[\SG]$ there exists $h_K\in H$ such that $\lexp{h_K}{\thetaS(K)}\in[\SH]$. We define $\thetaBS'\colon [\SG]\to[\SH]$ by $K\mapsto \lexp{h_K}\thetaS(K)$. The second part of Remark~\ref{old remark} implies that $\thetaBS'$ is a bijection. 
Next, for $K\in[\SG]$, we define $\theta'_K\colon \Hom(K,A)\to \Hom(\thetaBS'(K),A)$ by $\phi\mapsto \lexp{h_K}{\theta_K(\phi)}$, noting that $\thetaBS'(K)=\lexp{h_K}\thetaS(K)$. With $\theta_K$ being a bijection, also $\theta'_K$ is a bijection. By construction, we have $(\thetaBS'(K),\theta'_K(\phi))=\lexp{h_K}{(\thetaS(K),\theta_K(\phi))}$ for every $K\in[\SG]$ and $\phi\in\Hom(K,A)$.
Since the number $\gamma^H_{(R,\rho),(S,\sigma)}$ does not change when one replaces $(R,\rho)$ and $(S,\sigma)$ in $\MAH$ by any $H$-conjugate pairs, and since the original maps $\thetaS$ and $\theta_K$, for $K\le G$, satisfied condition (a) in Theorem~\ref{Previous Species Thm}, the bijections $\thetaS'$ and $\theta'_K$, for $K\in[\SG]$, now satisfy condition (a) in Theorem~\ref{SpeciesThm}.
Let $\widetilde{\Theta}\colon \widetilde{B^A}(G) \to \widetilde{B^A}(H)$ denote the ring isomorphism associated to the bijections $\thetaS$ and $\theta_K$ for $K\le G$ in Theorem~\ref{Previous Species Thm}(b) and let $\overline{\Theta'}\colon \qBAG \to \qBAH$ denote the map associated to the bijections $\thetaBS'$ and $\theta'_K$ for $K\in[\SG]$ in Theorem~\ref{SpeciesThm}(b). Then, for all $K\in[\SG]$ and all $\phi\in\Hom(K,A)$, we have
\begin{align*}
   \overline{\Theta'}\left(\pi_{[\SG]}\left(\widetilde{b}_{(K,\phi)}\right)\right) & = \overline{\Theta'}\left(\overline{b}_{(K,\phi)}\right) =
       \overline{b}_{(\thetaBS'(K),\theta'_K(\phi))} \\
       &= \overline{b}_{\lexp{h_K}{(\thetaS(K),\theta_K(\phi))}} = \pi_{[\SH]}\left(\widetilde{b}_{\lexp{h_K}{(\thetaS(K),\theta_K(\phi))}}\right) \\
       &= \pi_{[\SH]}\left(\widetilde{b}_{(\thetaS(K),\theta_K(\phi))}\right)= \pi_{[\SH]}\left(\widetilde{\Theta}(\widetilde{b}_{(K,\phi)})\right)\,.
 \end{align*} 
 Thus, $\overline{\Theta'}\circ\pi_{[\SG]}=\pi_{[\SH]}\circ\widetilde{\Theta}\colon \widetilde{B^A}(G)\to \overline{B^A}(H)$. Since $\widetilde{\Theta}$, $\pi_{[\SG]}\colon \widetilde{B^A}(G)\to\overline{B^A}(G)$ and $\pi_{[\SH]}\colon\widetilde{B^A}(H)\to\overline{B^A}(H)$ are ring isomorphisms, also $\overline{\Theta'}$ is a ring isomorphism.

\smallskip
Conversely, for each $[K,\phi]_G\in G\backslash\MAG$ we can assume $K\in [\SG]$, and mapping $[K,\phi]_G$ to $[\thetaBS(K),\theta_K(\phi)]_H$ gives a bijection from $G\backslash \MAG$ onto $H\backslash \MAH$. This bijection extends to an isomorphism of abelian groups $\Theta:\BAG\longrightarrow \BAH$. Now the diagram
$$\xymatrix{
\BAG\ar[rr]^-{\Theta}\ar[d]_{\overline{\Phi}^A_G} &&\BAH\ar[d]^{\overline{\Phi}^A_H}\\
\qBAG \ar[rr]_{\overline{\Theta}} &&\qBAH
}$$
commutes, and since the bottom and the vertical arrows are injective ring homomorphisms, also $\Theta$ is a ring isomorphism.
\qed

\bigskip
As remarked in \cite[Cor. 3.12]{Gar1}, some of the $\theta_K$ are necessarily group isomorphisms, but we ignore whether this has to be the case for all these maps. However, when the $\theta_K$ are isomorphisms, we can drop Condition~(b) in Theorem~\ref{SpeciesThm}.

\begin{prop}\thlabel{SpeciesCriterion}
Let $A$ be an abelian group and let $G$ and $H$ be finite groups. Assume that there is a bijection $\thetaBS:[\SG]\longrightarrow [\SH]$ and, for each $K\in[\SG]$, a group isomorphism $\theta_K\colon\Hom(K,A)\longrightarrow \Hom(\theta_{\mathcal{S}}(K),A)$ such that 
$$\gamma^H_{(\thetaBS(K),\theta_{K}(\phi)),(\thetaBS(L),\theta_{L}(\psi))} = \gamma^G_{(K,\phi),(L,\psi)}$$
for all $K,L\in [\SG]$, $\phi\in \Hom(K,A)$, $\psi\in \Hom(L,A)$. Then the assignment $[K,\phi]_G\mapsto [\thetaBS(K),\theta_K(\phi)]_H$ for $K\in [\SG]$ and $\phi\in\Hom(K,A)$ extends to a species isomorphism $\Theta\colon B^A(G) \longrightarrow B^A(H)$.
\end{prop}

\begin{proof}
For $K\in [\SG]$ and $\phi\in \Hom(K,A)$, the assignment $\overline{b}_{(K,\phi)}\mapsto \overline{b}_{(\thetaBS(K),\theta_K(\phi))}$ extends to an isomorphism $\overline{\Theta}:\overline{B^A}(G)\longrightarrow \overline{B^A}(H)$ of abelian groups. Then the diagram
$$\xymatrix{
\overline{B^A}(G)\ar@{^{(}->}[d]\ar[rr]^{\overline{\Theta}} &&\overline{B^A}(H)\ar@{^{(}->}[d]\\
\prod_{K\in[\SG]}\Z \Hom(K,A)\ar[rr]_{(\theta_{K})} &&\prod_{L\in[\SH]}\Z \Hom(L,A)\\
}$$
where $\theta_K\colon\Z \Hom(K,A)\longrightarrow \Z \Hom(\theta_{\mathcal{S}}(K),A)$ is the $\Z$-linear extension of $\theta_K$ and the vertical arrows are the inclusions, commutes by Remark~\ref{new remark}. Note that the bottom map is a ring isomorphism, since each $\theta_K$  was a group isomorphism. Therefore, as the bottom and the vertical arrows are injective, also $\overline{\Theta}$ is a ring isomorphism. By \thref{SpeciesThm} and its proof, the maps $\thetaBS$ and $\theta_K$ determine a species isomorphism.
\end{proof}

\section{Nontrivial counterexamples}

We recall the construction of Thévenaz' counterexamples in \cite{Thev}. Let $p$ and $q\geq 3$ be prime numbers such that $q|(p-1)$, and take elements $a\neq b$ of order $q$ in $(\Z/p\Z)^{\times}$. Let $P_a=\Z/p\Z=\langle x\rangle$, $P_b= \Z/p\Z=\langle y\rangle$ and $Q=C_q=\langle z\rangle$, then let $Q$ act on $P_a\oplus P_b$ by $\lexp{z}{x} = ax$ and $\lexp{z}{y}=by$, and consider the resulting semidirect product $G(a,b)=(P_a\oplus P_b)\rtimes Q$. A complete set of representatives of the conjugacy classes of subgroups of $G(a,b)$ is $\{1\}$, $P_a$, $P_b$, $P(j) = \langle x+jy\rangle$ for $j\in [(\Z/p\Z)^{\times}/\langle a\rangle]$, $P_a\oplus P_b$, $Q$, $P_a\rtimes Q$, $P_b\rtimes Q$ and $G(a,b)$. Taken in this order, the table of marks of $G(a,b)$ is independent of the choice of $\{a,b\}$. For fixed $p$ and $q$, there are precisely $\frac{q-1}{2}$ isomorphism classes of these groups; in fact, if also $c\neq d$ are elements of order $q$ in $(\Z/p\Z)^{\times}$ then $G(a,b)\cong G(c,d)$ if and only if there exists $n\in\{1.\ldots,q-1\}$ such that $\{c,d\}=\{a^n,b^n\}$ (see \cite{Thev}). Note that there are infinitely many choices for such $p$ and $q$: taking any prime $q\geq 3$, then by Dirichlet's theorem there are infinitely many primes $p$ in the arithmetic progression $1+q,1+2q,1+3q,\ldots$.

\begin{theorem} \thlabel{ThevenazGroups}
Let $p$ and $q\ge3$ be primes with $q$ dividing $p-1$ and let $a\neq b$ and $c\neq d$ be elements of order $q$ in $(\Z/p\Z)^{\times}$. If $A$ has trivial $p$-torsion, then $B^A(G(a,b))$ and $B^A(G(c,d))$ are isomorphic rings.
\end{theorem}

\begin{proof}
Since we already know that these groups have isomorphic Burnside rings, we can assume that $A$ has elements of order $q$. For simplicity, set $G:=G(a,b)$ and $H:=G(c,d)$ and take $[\SG]$ and $[\SH]$ as in the first paragraph of this section. We let $\thetaBS:[\SG]\longrightarrow [\SH]$ be the obvious bijection inducing an isomorphism of the tables of marks and set $K':=\thetaBS(K)$ for $K\in [\SG]$. Next we define group isomorphisms $\theta_K\colon\Hom(K,A)\longrightarrow \Hom(K',A)$ for $K\in[\SG]$: if $K$ is a $p$-subgroup so is $K'$, and $\Hom(K,A)$ and $\Hom(K',A)$ are trivial, hence there is only one choice for $\theta_K$; if $K$ is not a $p$-subgroup, then a homomorphism $\phi\colon K\longrightarrow A$ is determined by the value $\phi(z)$ which is either an element of order $q$ or $1$, and we define $\phi':=\theta_K(\phi)\colon K'\longrightarrow A$ by requiring $\phi'(z)=\phi(z)$.

\smallskip
We now compare $\gamma^G_{(K,\phi),(L,\psi)}$ and $\gamma^H_{(K',\phi'),(L',\psi')}$ for $K,L\in [\SG]$, $\phi\in \Hom(K,A)$ and $\psi\in \Hom(L,A)$. First, since $\thetaBS$ preserves the marks, we have 
$$\gamma^G_{(K,1),(L,1)}=|(G/L)^K|=|(H/L')^{K'}|=\gamma^H_{(K',1),(L',1)},$$
for all $K$ and $L$ in $[\SG]$. Note that the subgroups $P_a$, $P_b$ and $P_a\oplus P_b$ are normal, while $N_G(P(j))=P_a\oplus P_b$, and $P_a\rtimes Q$, $P_b\rtimes Q$ and $Q$ are self-normalizing. Moreover, it is straightforward to verify that, for $K$ and $L$ in $[\SG]$, if $K\not\leq L$ then $K\not\leq_{G} L$, and since $\thetaBS$ preserves containments, then $\gamma^G_{(K,\phi),(L,\psi)}=\gamma^H_{(K',\phi'),(L',\psi')}=0$ whenever $K\not\leq L$. 

\smallskip
Therefore, we are left with the case when $K\leq L$, $L$ is not a $p$-subgroup and $\psi\neq 1$, and we distinguish the following cases for $K$: 

\smallskip
(i) If $K\in\{\{1\}, P_a, P_b, P_a\oplus P_b\}$ then $K\unlhd G$ and $\phi=1$. Therefore $K\leq {\lexp{g}{L}}$ and $\lexp{g}{\psi}|_K=1$ for all $g\in G$ and hence
    $$\gamma^G_{(K,1),(L,\psi)}=[G:L]=[H:L']=\gamma^H_{(K',1),(L',\psi')}\,.$$
    
 \smallskip
 (ii) If $K=P(j)$ we may assume that $L=G$.  In this case, $\gamma_{(K,1),(G,\psi)}^G= 1 = \gamma_{(K',1),(H,\psi')}^H$.
    
 \smallskip
 (iii) If $K$ is not a $p$-subgroup, then $K,L\in\{Q, P_a\rtimes Q, P_b\rtimes Q, G\}$, and in particular $z\in K\le L$. It is straightforward to verify that in all cases for $L$ the following holds: if $g\in G$ with $g\notin L$ then $z\notin\lexp{g}{L}$ and therefore $K\not\leq \lexp{g}{L}$. 
We can conclude that
    $$\gamma^{G}_{(K,\phi),(L,\psi)}=\begin{cases}
    1 &\text{if}\;\phi(z)=\psi(z),\\
    0 &\text{otherwise},
    \end{cases}$$
    and by the way we have defined $\theta_K$, we have that $\gamma^{H}_{(K',\phi'),(L',\psi')}=\gamma^{G}_{(K,\phi),(L,\psi)}$.

\smallskip
We conclude that $\thetaBS$ and the isomorphisms $\theta_K$ for $K\in [\SG]$ satisfy the condition of \thref{SpeciesCriterion}, hence they determine a species isomorphism.
\end{proof}

\begin{rem}
The above theorem shows that Th\'evenaz' counterexamples to the isomorphism problem for the Burnside ring are also non-trivial counterexamples for the $A$-fibered Burnside ring if $A$ is any abelian group with trivial $p$-torsion and having elements of order $q$. In particular, we have a negative answer to the isomorphism problem of the $C_q$-fibered Burnside ring for any prime $q\geq 5$. Moreover, for any field $k$ of characteristic $p$, we obtain $D^k(G(a,b))\cong D^k(G(c,d))$, since in this case $D^k(G(a,b))\cong B^{C_q}(G(a,b))$. However, the existence of non-isomorphic finite groups $G$ and $H$ with $D^{\mathbb{C}}(G)\cong D^{\mathbb{C}}(H)$ remains open.
\end{rem}



\begin{thebibliography}{9}

\bibitem{Bolt} Robert Boltje.\ {\em A canonical Brauer Induction formula}. Astérisque 181-182 (1990), 31--59.

\bibitem{Bolt2} Robert Boltje.\ {\em A general theory of canonical induction formulae} 206 (1998), 293--343.

\bibitem{Boltje2001} Robert Boltje.\ {\em Monomial resolutions}. J.\ Algebra 246 (2001), 811--848.

\bibitem{BY} Robert Boltje and Deniz Y\i lmaz.\ {\em The $A$-fibered Burnside ring as $A$-fibered biset functor in characteristic zero.} Algebr.\ Represent.\ Theory 24 (2021), 1359--1385.

\bibitem{Dress} Andreas Dress.\ {\em The ring of monomial representations. I. Structure theory.}  J.\ Algebra 18 (1971), 137--157. 

\bibitem{Gar1} Benjamín García.\ {\em Species isomorphisms of fibered Burnside rings.} Comm.\ Algebra 51:3 (2023), 949--957, DOI: 10.1080/00927872.2022.2117818

\bibitem{Thev} Jacques Th\'evenaz.\ {\em Isomorphic Burnside rings}. Comm.\ Algebra 16 (1988), 1945--1947.

\end{thebibliography}
\end{document}